\documentclass{amsart}
\usepackage{amsfonts}
\usepackage{color}
\def\R{{\mathbb {R}}}
\def\A{{\mathcal {A}}}
\def\N{{\mathbb {N}}}

\def\G{{\mathcal{G}}}
\def\L{{\mathcal{L}}}
\def\K{{\mathcal{K}}}

\def\Ks{{\mathcal{K}_{\mathrm{sym}}}}
\def\dh{d_\mathcal{H}}

\newtheorem{teo}{Theorem}[section]
\newtheorem{lema}[teo]{Lemma}
\newtheorem{prop}[teo]{Proposition}

\theoremstyle{remark}
\newtheorem{remark}[teo]{Remark}

\theoremstyle{definition}
\newtheorem{defi}[teo]{Definition}

\numberwithin{equation}{section}

\title[Eigenvalues of nonlocal problems]{Gamma convergence and asymptotic behavior for eigenvalues of nonlocal problems}
\author[J. Fern\'andez Bonder, A. Silva and J.F. Spedaletti]{Juli\'an Fern\'andez Bonder , Anal\'ia Silva and Juan F. Spedaletti}

\address[J. Fern\'andez Bonder]{Instituto de Matem\'atica Luis A. Santal\'o (IMAS), CONICET\hfill\break\indent
Departamento de Matem\'atica, FCEN - Universidad de Buenos Aires
\hfill\break \indent Ciudad Universitaria, Pabell\'on I, C1428EGA, Av. Cantilo s/n. \hfill\break \indent Buenos Aires, Argentina.}

\email{jfbonder@dm.uba.ar}
\urladdr{http://mate.dm.uba.ar/~jfbonder}

\address[A. Silva]{Instituto de Matem‡tica Aplicada San luis (IMASL) \hfill\break\indent
Universidad Nacional de San Luis, CONICET \hfill\break\indent
Ejercito de los Andes 950, D5700HHW, \hfill\break\indent
San Luis, Argentina}

\email{acsilva@unsl.edu.ar}
\urladdr{https://analiasilva.weebly.com/}

\address[J.F. Spedaletti]{Instituto de Matem‡tica Aplicada San luis (IMASL) \hfill\break\indent
Universidad Nacional de San Luis, CONICET \hfill\break\indent
Ejercito de los Andes 950, D5700HHW, \hfill\break\indent
San Luis, Argentina}

\email{jfspedaletti@unsl.edu.ar}

\begin{document}

\begin{abstract}
In this paper we analyze the asymptotic behavior of several fractional eigenvalue problems by means of Gamma-convergence methods. This method allows us to treat different eigenvalue problems under a unified framework. We are able to recover some known results for the behavior of the eigenvalues of the $p-$fractional laplacian when the fractional parameter $s$ goes to 1, and to extend some known results for the behavior of the same eigenvalue problem when $p$ goes to $\infty$. Finally we analyze other eigenvalue problems not previously covered in the literature.
\end{abstract}

%35J92  	Quasilinear elliptic equations with $p$-Laplacian
%49R05  	Variational methods for eigenvalues of operators
%35P30  	Nonlinear eigenvalue problems, nonlinear spectral theory


\maketitle 

\section{Introduction}
The goal of this paper is to study the asymptotic behavior of several eigenvalue problems for nonlocal (and nonlinear) operators under a unified approach given by the Gamma-convergence.

Up to our knowledge, this approach was first applied in a paper by T. Champion and L. De Pascale in \cite{ChDP}. In that paper the authors prove, in the context of eigenvalue problems for the $p-$Laplacian, an abstract result that lead them to show in a unified framework some asymptotic behavior for the eigenvalues of some $p-$Laplacian type problems that included the cases where $p\to\infty$ and some homogenization results. See also \cite{Bonder-Pinasco-Salort} for some results in the case of indefinite weights.

In recent years, there have been a great deal of work in understanding nonlocal phenomena and as a consequence,  the so-called {\em fractional eigenvalues}. See \cite{BPS, DRS, LL}, etc.

For instance, in \cite{LL} the authors considered the eigenvalues associated to the fractional $p-$Laplacian and analyze their limit as $p\to\infty$. For this problem, one of the results in \cite{LL} proves that the first fractional eigenvalue converges as $p\to \infty$ to the first eigenvalue of the {\em H\"older infinity Laplacian}. 

On the other hand, in \cite{BPS} the authors analyze the limit as $s\to 1$ of the same fractional eigenvalue problem and prove that these eigenvalues converge to the eigenvalues of the local $p-$Laplace operator. See also \cite{FBS} for a related result.

So the main purpose of this work is to extend the general abstract result of \cite{ChDP} to the fractional setting and show that all of the above mentioned results can be easily deduced from this theory. Finally, as a further example of the flexibility of our results, we show some new asymptotic behavior for eigenvalues not previously covered in the literature.

\subsection*{Organization of the paper}
The paper is organized as follows: In section \S 2, we recall all the preliminaries needed in the course of the work. This includes the concepts of Gamma convergence, Krasnoselskii genus, Hausdorff distance of compact sets and fractional order Sobolev spaces. All the results in this section are well known to experts and the reader can safely skip it and return to it only if necessary.

In section \S 3, we prove the main result of the paper, Theorem \ref{main.teo} that is the general abstract result on convergence of eigenvalues.

Finally, in section \S 4, we apply our abstract result to a number of examples where we recover some results of \cite{BPS}, extend the analysis of \cite{LL} and analyze some other eigenvalue problems not previously considered in the literature.

\section{Preliminaries}
In this section we review some elementary facts needed in the course of our main results. All the results presented here are well known to experts and you can safely skip this section and return to it when necessary. The topics reviewed in the section are: Gamma convergence, Krasnoselskii genus, fractional order Sobolev spaces and Hausdorff distance for compact sets. We divide this section into four subsections each one corresponding to an specific topic.

\subsection{Gamma convergence}
Let us begin with the basic definition.

\begin{defi}\label{def.gamma}
Let $(X,d)$ be a metric space and let $\{F_n\}_{n\in\N}$ be a sequence of functions defined in $X$ into the extended real line. We say that $F_n$ $\Gamma-$converges to some function $F_0$ (denoted by $F_n \stackrel{\Gamma}{\to} F_0$) if the following two inequalities hold:
\begin{itemize}
\item ($\liminf-$inequality)

For every $x\in X$ and for every $\{x_n\}_{n\in\N}\subset X$ such that $x_n\to x$ in $X$, it holds
\begin{equation}\label{liminf}
F_0(x)\le \liminf_{n\to\infty} F_n(x_n).
\end{equation}

\item ($\limsup-$inequality)

For every $x\in X$, there exists $\{y_n\}_{n\in\N}\subset X$ such that $y_n\to x$ in $X$ and
$$
F_0(x)\ge \limsup_{n\to\infty} F_n(y_n).
$$
\end{itemize}
\end{defi}

\begin{remark} The sequence $\{y_n\}_{n\in\N}$ given above is usually called the {\em recovery sequence}.
\end{remark}

The concept of Gamma convergence goes back to the 60s and was introduced by E. De Giorgi. It is the natural concept to deal with the approximation of minimization problems. The next result is not the most general one about Gamma convergence but it will suit our purposes. For a throughout introduction to the subject, we recommend the interested reader the excellent book \cite{DalMaso}.

\begin{teo}\label{teo.gamma}
Let $(X,d)$ be a metric space, $\{F_n\}_{n\in\N}$ be a sequence of functions defined in $X$ into the extended real line such that $F_n\stackrel{\Gamma}{\to} F_0$ for some function $F_0$. Assume that there exists $x_n\in X$ such that
$$
F_n(x_n)=\inf_X F_n = \min_X F_n.
$$
Moreover, assume that the sequence $\{x_n\}_{n\in\N}$ is precompact in $X$. Then
$$
\inf_X F_0 = \lim_{n\to\infty} \left(\inf_X F_n\right)
$$
and for every accumulation point $x_0$ of the sequence $\{x_n\}_{n\in\N}$, we have that
$$
F_0(x_0)=\inf_X F_0 =\min_X F_0.
$$
\end{teo}

The proof of Theorem \ref{teo.gamma} is elementary and, of course, is by no means the best result in Gamma convergence. As we mentioned before, we refer to \cite{DalMaso} for sharper results, but for the purpose of this paper, Theorem \ref{teo.gamma} will be enough in most parts.

A useful generalization, again with an elementary proof, is the following.
\begin{teo}\label{teo.gamma.2}
Let $(X,d)$ be a metric space, $\{F_n\}_{n\in\N}$ be a sequence of functions defined in $X$ into the extended real line.

Assume that there exists $F_0\colon X\to\bar\R$ such that the $\liminf-$inequality \eqref{liminf} holds and assume in addition that the sequence $\{F_n\}_{n\in\N}$ is equicoercive and 
$$
\inf_X F_0 = \lim_{n\to\infty} \left(\inf_X F_n\right).
$$

Let $\{x_n\}_{n\in\N}\subset X$ be any sequence of quasi-minima of $\{F_n\}_{n\in\N}$, i.e.
$$
F_n(x_n) = \inf_X F_n + o(1).
$$
Then, any accumulation point $x_0$ of $\{x_n\}_{n\in\N}$ is a minimum point of $F_0$.
\end{teo}

\subsection{Krasnoselskii's genus and its properties}

The Krasnoselskii's genus is a widely used topological tool in minimax method for finding solutions to some PDE problems. In particular, it is the most common tool in nonlinear eigenvalue theory. We refer to \cite{Rabinowitz} for a very good introduction to the subject.

In this very short subsection we recall its definition and the properties required in our proofs.

\begin{defi}
Let $E$ be a real Banach space and $A\subset E$ be a nonempty closed and symmetric set, i.e. $A=-A$.

The genus of $A$, denoted by $\gamma(A)$, is defined as
$$
\gamma(A) := \inf\{m\in\N\colon \text{there exists a continuous and odd function } \varphi\colon A\to \R^m\setminus\{0\}\}.
$$
In the definition above, it is understood that $\inf\emptyset = \infty$.
\end{defi}

Among its many interesting properties, the following one is the one that is going to be used explicitly in the paper.

\begin{prop}\label{prop.genus}
Assume that $A\subset E$ is symmetric and compact. Then, there exists a symmetric neighborhood $N$ of $A$ such that $\gamma(\bar{N}) = \gamma(A)$.
\end{prop}

\begin{remark}
Observe that we trivially have that $A\subset B\Rightarrow \gamma(A)\le \gamma(B)$. So Proposition \ref{prop.genus}  says that for compact symmetric sets, we can slightly fatten the set without increasing the genus.
\end{remark}

\subsection{Fractional order Sobolev spaces}
In this subsection, we recall the definition and basic properties of fractional order Sobolev spaces. Even though this is a very classical topic, its interest in PDEs has been renewed in the last 20 years or so, and a very good reference for the subject is the review article \cite{DNPV}.

Given a function $u\in L^1_{\rm loc}(\R^N)$, $0<s<1\le p<\infty$, we define its Gagliardo semi-norm as
$$
[u]_{s,p} := \left(\iint_{\R^N\times\R^N} \frac{|u(x)-u(y)|^p}{|x-y|^{N+sp}}\, dxdy\right)^\frac{1}{p}.
$$
We then define the fractional order Sobolev space $W^{s,p}(\R^N)$ as
$$
W^{s,p}(\R^N) := \{u\in L^p(\R^N)\colon [u]_{s,p}<\infty\}.
$$
This space is a Banach space with norm given by $\|u\|_{s,p} = (\|u\|_p^p + [u]_{s,p}^p )^\frac{1}{p}$. This is a separable space and reflexive for $p>1$.

Given an open set $\Omega\subset\R^n$, we then consider
\begin{equation}\label{def.fraccionario}
W^{s,p}_0(\Omega) := \{u\in W^{s,p}(\R^N)\colon u=0 \text{ a.e. in } \R^N\setminus\Omega\}.
\end{equation}
Since $W^{s,p}_0(\Omega)$ is a closed subspace of $W^{s,p}(\R^N)$ it inherits the properties of separability and reflexivity.

\begin{remark}
This definition for $W^{s,p}_0(\Omega)$, is not the most natural one. One can instead consider the Gagliardo semi-norm defined in $\Omega$ as
$$
[u]_{s,p;\Omega} := \left(\iint_{\Omega\times\Omega} \frac{|u(x)-u(y)|^p}{|x-y|^{N+sp}}\, dxdy\right)^\frac{1}{p}.
$$ 
Then define $W^{s,p}(\Omega) := \{u\in L^p(\Omega)\colon [u]_{s,p;\Omega}<\infty\}$ with norm $\|u\|_{s,p;\Omega} =  (\|u\|_{p;\Omega}^p + [u]_{s,p;\Omega}^p )^\frac{1}{p}$ and hence define $W^{s,p}_0(\Omega)$ as the closure of $C^\infty_c(\Omega)$ with respect to $\|\cdot\|_{s,p;\Omega}$.

This two definitions for $W^{s,p}_0(\Omega)$ are known to coincide if, for instance, $\Omega$ has Lipschitz boundary or if $0<s<\tfrac{1}{p}$. In this latter case, $W^{s,p}_0(\Omega)=W^{s,p}(\Omega)$.
\end{remark}

In this paper, we consider the space $W^{s,p}_0(\Omega)$ given by \eqref{def.fraccionario}.

For functions in this space, Poincar\'e inequality holds true.
\begin{prop}\label{Poincare}
There exists a constant $C$ depending on $N, s, p$ and $|\Omega|$ such that
$$
\|u\|_p\le C [u]_{s,p},
$$
for every $u\in W^{s,p}_0(\Omega)$.
\end{prop}

As usual, Poinar\'e inequality implies that $[\,\cdot\,]_{s,p}$ defines a norm in $W^{s,p}_0(\Omega)$ equivalent to $\|\cdot\|_{s,p}$.

The following theorem gives the Sobolev immersion in the fractional setting.
\begin{teo}\label{Sobolev.immersion}
Assume that $\Omega$ has finite measure and that $sp\le N$. Define the critical Sobolev exponent as
$$
p_s^* := \begin{cases}
\frac{Np}{N-sp} & \text{if } sp<N\\
\infty & \text{if } sp=N.
\end{cases}
$$
Then, for any $1\le q\le p_s^*$ (or $1\le q<\infty$ if $sp=N$) we have that
$$
\|u\|_q\le C [u]_{s,p},
$$
for some constant $C$ depending on $N, s, p, q$ and $\Omega$.

Moreover, if $1\le q < p_s^*$, then the immersion $W^{s,p}_0(\Omega)\subset L^q(\Omega)$ is compact.
\end{teo}

When $sp>N$ we have Morrey-type estimates. That is
\begin{teo}\label{Morrey}
Let $\Omega$ be bounded and $sp>N$. Then any $u\in W^{s,p}_0(\Omega)$ admits a continuous representative and
$$
[u]_\alpha \le C [u]_{s,p},
$$
for some constant $C$ depending on $N, s, p$ and $\Omega$. Here $\alpha=s-\tfrac{N}{p}$ and $[\, \cdot\, ]_\alpha$ is the H\"older semi-norm, defined as
$$
[u]_\alpha:= \sup_{x, y\in \Omega} \frac{|u(x)-u(y)|}{|x-y|^\alpha}.
$$
\end{teo}

The proof of all the above mentioned results on fractional order Sobolev spaces, can be found, for instance, in \cite{DNPV}.

To describe the behavior of the spaces $W^{s,p}$ when the fractional parameter $s$ grows to 1 we use the celebrated results of \cite{BBM}. These results can be summarized as follows
\begin{teo}\label{gamma.BBM}
Let $\Omega$ be bounded and $1<p<\infty$. For any $0<s<1$ we define $F_s\colon L^p(\Omega)\to \bar\R$ as
$$
F_s(u) = \begin{cases}
(1-s)[u]_{s,p}^p & \text{if } u\in W^{s,p}_0(\Omega),\\
\infty & \text{otherwise}.
\end{cases}
$$
Moreover, consider $F_0\colon L^p(\Omega)\to \bar\R$ as
$$
F_0(u) = \begin{cases}
K \|\nabla u\|_p^p & \text{if } u\in W^{1,p}_0(\Omega),\\
\infty & \text{otherwise},
\end{cases}
$$
where 
$$
K = K(N,p) =\frac{1}{p}\int_{|z|=1} |z_N|^p\, dS.
$$
Then for every sequence $s_n\to 1$, it follows that $F_{s_n}\stackrel{\Gamma}{\to} F_0$.
\end{teo}

See also \cite{Ponce} for more results regarding the connection of fractional norms and Gamma convergence.

\begin{remark}
In \cite{BBM} it is also shown that the functionals $F_s$ in fact converges pointwise to $F_0$, so in the definition of Gamma convergence one can take as a recovery sequence $\{u_n\}_{n\in\N}$ a constant sequence.
\end{remark}

\subsection{Hausdorff distance}

Let $(X,d)$ be a metric space and consider 
$$
\K = \K(X) = \{K\subset X\colon K \text{ compact}\}.
$$ 
In $\K$ we define the Hausdorff distance, as $\dh\colon \K\times\K\to \R$
$$
\dh(A, B) = \max\left\{\sup_{x\in A} d(x, B); \sup_{x\in B} d(x, A)\right\}.
$$

\begin{remark}\label{remK}
It is a well known fact that $(\K, \dh)$ is a metric space and that $(\K, \dh)$ is compact if $(X,d)$ is compact.
\end{remark}

\begin{remark}\label{remK2}
Observe that if $Y\subset X$ is a subspace, then $\K(Y)$ is a subspace of $\K(X)$. Hence, an important consequence of Remark \ref{remK} is that  if $Y$ is compact, then $\K(Y)= \{K\in \K(X)\colon K\subset Y\}$ is a compact set of $\K(X)$.
\end{remark}

\begin{remark}
Two particular important examples of the situation described in Remark \ref{remK2} are the followings:
\begin{itemize}
\item Take $1\le sq\le N$ and $1\le p<q_s^*$. Then consider $X=L^p(\Omega)$ and $Y=\{u\in W^{s,q}_0(\Omega)\colon [u]_{s,q}\le C\}$ for some constant $C>0$. Then, by Theorem \ref{Sobolev.immersion}, $Y\subset X$ is compact. Hence $\K(Y)$ is a compact subspace of $\K(X)$.

\item Take $sq>N$, $X=C_0(\Omega)$ (the closure of $C_c(\Omega)$ with respect to $\|\cdot\|_\infty$) and $Y=\{u\in W^{s,q}_0(\Omega)\colon [u]_{s,q}\le C\}$ for some constant $C>0$. Then, by Theorem \ref{Morrey}, $Y\subset X$ is compact and hence $\K(Y)$ is a compact subspace of $\K(X)$.
\end{itemize}
\end{remark}

Then next proposition, that characterizes the convergence in $\dh$, will be most useful. 
\begin{prop}\label{conv.dh}
Let $\{K_n\}_{n\in\N}\subset \K$ and $K\in \K$. Then $K_n\stackrel{\dh}{\to} K$ if and only if
\begin{enumerate}
\item for each sequence $\{x_n\}_{n\in\N}\subset X$, such that $x_n\in K_n$ for all $n\in\N$, any accumulation  point $x\in X$ of $\{x_n\}_{n\in\N}$ belongs to $K$,

\item for each $x\in K$, there exists $\{x_n\}_{n\in\N}\subset X$ such that $x_n\in K_n$ for all $n\in\N$ and $x_n\to x$.
\end{enumerate}
\end{prop}

The proof of this proposition is a straightforward consequence of the definitions.

Suppose now that $X$ is also a Banach space. Then consider the subspace of $\K$ of {\em symmetric} compact subsets, i.e.
$$
\Ks := \{K\in\K\colon K=-K\}.
$$
An immediate consequence of Proposition \ref{conv.dh} is that $\Ks$ is a closed subspace of $\K$.

\section{The abstract result}

Let $\{F_n\}_{n\in\N}$ be a family of functionals defined in $L^1(\Omega)$ with values  in $[0,\infty]$ such that:
\begin{enumerate}
\item[(A1)]\label{A1} For each $n\in\N$, $F_n$ is convex and 1-homogeneous.

\item[(A2)]\label{A2} There exist two constants $0<\alpha<\beta$ such that for any $n\in\N$ there exist $p_n\in[1,\infty)$  and $s_n\in(0,1)$, such that
\begin{align*}
\alpha(1-s_n)^\frac{1}{p_n}[v]_{s_n,p_n}\leq &F_n(v)\leq\beta(1-s_n)^\frac{1}{p_n}[v]_{s_n,p_n}\mbox{  if } v\in W_0^{s_n,p_n}(\Omega),\\
& F_n(v)=+\infty\mbox{ otherwise.}
\end{align*}
%(When $s_n=1$ one has to consider $(1-s_n)[v]_{s_n,p_n} = K\|\nabla v\|_{p_n}$).

\item[(A3)] There exists $p_0 \in [1,\infty]$ and $s_0\in (0,1]$ such that the sequences $\{p_n\}_{n\in\N}$ and $\{s_n\}_{n\in\N}$ given in (A2) verify that $p_n\to p_0$ and $s_n\to s_0$ as $n\to\infty$. Moreover, there exists $F_0\colon L^1(\Omega)\to [0,\infty]$ such that  $F_n$ $\Gamma-$converges in $L^{p_0}(\Omega)$ (or $C_0(\Omega)$ if $p_0=\infty$) to  $F_0$.
\end{enumerate}

We define
\begin{align*}
\G^{k}_{s,p}=\left\{G\subset W_0^{s,p}(\Omega)\colon \begin{array}{l} G=-G \mbox{ closed and
bounded in }W_0^{s,p}(\Omega) \\
\|u\|_p=1,\ \forall u\in G \mbox{ and }\gamma(G)\geq k\end{array}\right\}.
\end{align*}
For any $k\in\N$ we associate with $F_n$ the functional
$J_n^{k}\colon \Ks(\Omega)\to[0,\infty]$ given by
$$
J_n^k(G) := \begin{cases}
\sup_{v\in G} F_n(v), & G\in \G^k_{s_n,p_n}; \\
+\infty, & \text{otherwise.} 
\end{cases}
$$
where $\Ks(\Omega)$ is the set of compact symmetric subsets of $L^{p_0}(\Omega)$ (or $C_0(\Omega)$ if $p_0=+\infty$).

We define the $k-$th eigenvalue of the functional $F_n$ as
$$
\lambda_n^k := \inf_{G\in \Ks(\Omega)} J_n^{k}(G),
$$

So, the main result of the paper is the following.
\begin{teo}\label{main.teo}
Let $k$ be a positive integer, and assume that $\{F_n\}_{n\in\N}$ satisfies {\em (A1)-(A3)}. Then the sequence $\{J_n^{k}\}_{n\in\N}$ is equicoercive, it verifies the $\liminf$ inequality \eqref{liminf}  and
$$
\lim_{n\to\infty}\left(\inf_{G\in \Ks(\Omega)} J_n^{k}(G)\right) = \inf_{G\in \Ks(\Omega)}  J_0^{k}(G).
$$
\end{teo}

This result is a generalization of the result of \cite{ChDP} to the fractional setting. The strategy of the proof very much resembles the one in the above mentioned paper. The main modification is done in the first step of the proof.

Nevertheless, we include full details of the proof for the reader's convenience.

\begin{proof}
The proof is divided into three steps.

\medskip

{\em Step1:} equicoercivity. 
\smallskip

Assume first that $p_0<\infty$. If $G\in\{J_n^k \leq\mu\}$ then $G\in \G^{k}_{s_n,p_n}$ and 
\begin{equation}\label{Fnmu}
F_n(v)\leq\mu \quad \text{for all } v\in G. 
\end{equation}

Assume that $p_0\le N$ therefore, for $\delta>0$ small enough we take $q=p_0-\delta$ and $0<t<s_0$ such that
\begin{equation}\label{compact}
q_t^* := \frac{Nq}{N-tq}>p_0.
\end{equation}
Therefore, there exists $n_0\in\N$ such that $t<s_n$ and $q\le p_n$ for every $n\ge n_0$. So we have that $W^{s_n,p_n}_0(\Omega)\subset W^{t,q}_0(\Omega)$ and there exists a constant $C$ depending on $t, q$ and $\Omega$ but independent on $n\ge n_0$ such that
\begin{equation}\label{bbm1}
(1-t)^\frac{1}{q}[v]_{t,q}\le C (1-s_n)^\frac{1}{p_n}[v]_{s_n,p_n},
\end{equation}
for every $n\ge n_0$. See \cite{BBM}.

Using (A2), \eqref{Fnmu} and \eqref{bbm1} we have $(1-t)^\frac{1}{q}[v]_{t,q}\leq C$ for every $v\in G$ where $C$ depends only on $\mu, t, q$ and $\Omega$. We define
$$
K := \{v\in L^{p_0}(\Omega): (1-t)^\frac{1}{q}[v]_{t,q}\leq C\}.
$$
By the Sobolev embedding, from \eqref{compact}, we obtain that $K$ is compact in $L^{p_0}(\Omega)$, so $\{J_n^k \leq\mu\}\subset \{G\in \Ks(\Omega)\colon G\subset K\}$ which is a compact subset of $\Ks(\Omega)$, so that the family $J_n^{k}$ is equicoercive.

If $p_0>N$, we can take $N<q<p_0$ and use the exact same argument using Morrey's inequality (see \cite{DNPV}) instead of the Sobolev embedding.

The case with $p_0=\infty$ follows again easily from Morrey's inequality.

\medskip

{\em Step 2:}  $\Gamma-\liminf$.
\smallskip

 Let $G_0\in \Ks(\Omega)$ and $\{G_n\}_{n\in\N}\subset \Ks(\Omega)$ be a sequence such that $G_n\to G_0$
 in Hausdorff distance. We shall prove that
$$\liminf_{n\to\infty} J_n^{k}(G_n)\geq J_0^{k}(G_0).
$$
 
Without loss of generality, we may assume that there exists a constant $C>0$ such that
 $J_n^{k}(G_n)<C$ for every $n\in\N$. Observe that this implies that $G_n\in\G^k_{s_n,p_n}$. 
 
We will show $\gamma(G_0)\geq k$. To this end, take an open neighborhood $N$ of $G_0$ in $L^{p_0}(\Omega)$ such that $\gamma(G_0)=\gamma(\overline{N})$. Since $G_n\to G_0$ in Hausdorff distance, $G_n\subset N$ for $n$ large enough. Therefore, by the monotonicity of the genus we get
 $$
 k\leq\gamma(G_n)\leq\gamma(\overline{N})=\gamma(G_0).
$$
Now, for any $u\in G_0$ there exists a sequence $u_n\in G_n$ such that
$u_n\to u$ in $L^{p_0}(\Omega)$. By assumption (A3) we have that
$$
F_0(u)\leq\liminf_{n\to\infty}
F_n(u_n)\leq\liminf_{n\to\infty}\sup_{v\in G_n}F_n=\liminf_{n\to\infty} J_n^{k}(G_n)
$$
for all $u\in G_0$. Taking supremum we obtain the desired result.

\medskip

{\em Step 3:} It only remains to prove that
$$
\limsup_{n\to\infty} \left(\inf_{G\in \Ks(\Omega)} J_n^k(G)\right) \leq \inf_{G\in \Ks(\Omega)}J_0^{k}(G).
$$
Assume first that $p_0>1$. 

We fix $\delta>0$ small and assume first that $p_0<\infty$. Let $G_0\in \Ks(\Omega)$ be such that
$$
\inf_{G\in \Ks(\Omega)}J_0^{k}(G)\geq J_0^{k}(G_0)-\delta
$$
Since $G_0$ is compact in $L^{p_0}(\Omega)$, there exist $u^1,u^2,\dots,u^m\in G_0$ such that 
$$
G_0\subset \bigcup_{i=1}^{m} B_{L^{p_0}(\Omega)}(u^i, \tfrac15\delta).
$$ 
Since $F_n\stackrel{\Gamma}{\to} F_0$, there exist $\{u^{i}_n\}_{n\in\N}\subset L^{p_0}(\Omega)$ such that  $u^{i}_n\to u^{i}$ in $L^{p_0}(\Omega)$ as $n\to\infty$ and $F_n(u^{i}_n)\to F_0(u^{i})$ as $n\to\infty$ for every $i=1,\dots,m$. In particular, this implies that $u^i_n\in W^{s_n,p_n}_0(\Omega)$ for every $n\in\N$ and every $i=1,\dots,m$.

Next, for every $n\in \N$, we define
$$
C_n=\overline{Co(\{\pm u^{i}_n; i=1,\dots,m\})}.
$$
Note that $C_n\subset L^{p_0}(\Omega)$ is compact. 

Observe that from the Dominated Convergence Theorem,
$$
\lim_{q\uparrow p_0}\int_\Omega |v|^q\,dx =\int_\Omega |v|^{p_0}\,dx
$$ 
for every $v\in L^{p_0}(\Omega)$. Therefore, there exists $q_0<p_0$ such that
$$
\|u^i\|_q\ge \|u^i\|_{p_0} - \tfrac15\delta =  1-\tfrac{1}{5}\delta \mbox{ for  } i=1,\dots, m
$$
for every $q_0<q<p_0$.

 We denote by $\Pi_n$ the projection onto $C_n$, for the norm of $L^{q}(\Omega)$.  That is $\Pi_n\colon L^{q}(\Omega)\to L^{q}(\Omega)$ such that $\Pi_n v\in C_n$ and
$$
\|\Pi_n v - v\|_q\le  \inf_{u\in C_n} \|u-v\|_q.
$$
Observe that $\Pi_n$ is well defined since $C_n$ is compact and $\|\cdot\|_q$ is uniformly convex. Moreover, $\Pi_n$ is Lipschitz with Lipschitz constant 1.

Now, we want to prove that $\Pi_n(G_0)$ is far from $0$. To this end, if $v\in G_0$, we have that there exists $i\in \{1,\dots,m\}$ such that $\|v-u^i\|_{p_0}< \tfrac15\delta$. Therefore
$$
\|\Pi_n v\|_q \ge \|u^{i}_n\|_q - \|\Pi_n u^{i} - u^{i}_n\|_q - \|\Pi_n v - \Pi_n u^{i}\|_q.
$$
First, we compute the last term
$$
\|\Pi_n v - \Pi_n u^{i}\|_q \le \|v-u^{i}\|_q \le \|v-u^i\|_{p_0} |\Omega|^{\frac{1}{q}-\frac{1}{p_0}}<\tfrac15\delta  |\Omega|^{\frac{1}{q}-\frac{1}{p_0}}.
$$
For the second term, 
$$
\|\Pi_n u^{i} - u^{i}_n\|_q\le \|\Pi_n u^{i} - u^{i}\|_q+ \|u^{i}-u^{i}_n\|_q\le 2\|u^{i}-u^{i}_n\|_q\to 0 \quad \text{as }n\to\infty.
$$
The first term
$$
\|u^{i}_n\|_q\to\|u^{i}\|_q>1-\tfrac{1}{5}\delta.
$$
Hence, choosing $q$ close enough to $p_0$ and $n$ large enough, we obtain
$$
\|\Pi_n v\|_q\geq 1-\tfrac12 \delta \text{ for all } v\in G_0.
$$
So, $\Pi_n(G_0)\subset C_n-B(0,1-\frac{\delta}{2}).$ On the other hand, since $C_n\in \Ks(\Omega)$, it follows that $\Pi_n(G_0)\in \Ks(\Omega)$ and $\gamma(\Pi_n(G_0))\geq k$, therefore if we define
$$
G_n:= \left\{\frac{v}{\|v\|_{p_n}}\colon v\in \Pi_n(G_0)\right\},
$$
we get that $G_n\in \G^k_{s_n,p_n}(\Omega)$.

Now, using that $1-\frac{\delta}{2}\leq\|v\|_q\leq\|v\|_{p_n}|\Omega|^{\frac{1}{q}-\frac{1}{p_n}}$, for every $v\in \Pi_n(G_0)\subset C_n$, we obtain
\begin{align*}
F_n\left(\frac{v}{\|v\|_{p_n}}\right)&=\frac{1}{\|v\|_{p_n}}F_n(v) \le\frac{|\Omega|^{\frac{1}{q}-\frac{1}{p_n}}}{1-\frac{\delta}{2}}F_n(v)\\
&\leq\frac{|\Omega|^{\frac{1}{q}-\frac{1}{p_n}}}{1-\frac{\delta}{2}}\sup_{C_n}F_n \leq\frac{|\Omega|^{\frac{1}{q}-\frac{1}{p_n}}}{1-\frac{\delta}{2}}\max_{1\leq i\leq m}F_n(u^{i}_m).
\end{align*}
So,
$$
J_n^{k}(G_n)\leq\frac{|\Omega|^{\frac{1}{q}-\frac{1}{p_n}}}{1-\frac{\delta}{2}}\max_{1\leq i\leq
m}F_n(u^{i}_m).
$$ 
As a consequence,
\begin{align*}
\limsup_{n\to\infty}(\inf J_n^{k})&\leq \limsup_{n\to\infty}
J_n^{k}(G_n)\\
&\leq\frac{|\Omega|^{\frac{1}{q}-\frac{1}{p_0}}}{1-\frac{\delta}{2}}\max_{1\leq
i\leq m}F_0(u^{i})\\
&\leq\frac{|\Omega|^{\frac{1}{q}-\frac{1}{p_0}}}{1-\frac{\delta}{2}}\sup_{
G_0}F_0\\
&\leq\frac{|\Omega|^{\frac{1}{q}-\frac{1}{p_0}}}{1-\frac{\delta}{2}}(\inf
J_0^{k}+\delta).
\end{align*}
The conclusion of step 3  follows by letting $\delta$ go to $0$ and $q$ go to $p_0$.

The proof for the case $p_0=1$ is almost identical and the reader can check the details in \cite{ChDP}.
\end{proof}

\begin{remark}
In general it is not possible to show that the sets $G_n$ constructed in Step 3, converge in $\Ks$ to $G_0$, and that is why from the arguments given in the previous proof one cannot infer that $J_n^k$ $\Gamma-$converges to $J_0^k$.
\end{remark}

\section{Applications}

In this section we apply our abstract result to a large variety of eigenvalue problems. 

Some of the problems were previously consider in the literature and are fully understood. In that case we show how one can derive the same type of result with our general approach. 

In other cases, the problem has been treated before, but only partial answers are known prior to this work. In that case, by applying this general framework, we are able to fully analyze the convergence of the eigenvalues.

Finally, we also consider some problems that, up to our knowledge, were not consider previously.

\subsection{Eigenvalues of the fractional $p-$Laplacian}
The fractional $p-$laplacian operator is (formally) defined as 
$$
(-\Delta_p)^s u(x) = (1-s)\text{ p.v.}\int_{\R^N}\frac{|u(x)-u(y)|^{p-2}(u(x)-u(y))}{|x-y|^{N+sp}}\, dy.
$$
This operator can be seen as the gradient of the Gagliardo semi-norm, in the sense that, if
$$
\Phi_{s,p}\colon W^{s,p}_0(\Omega)\to \R,\qquad \Phi_{s,p}(u) := (1-s)[u]_{s,p}^p,
$$
then $(-\Delta_p)^s\colon W^{s,p}_0(\Omega)\to W^{-s,p'}(\Omega) := (W^{s,p}_0(\Omega))^*$ (the dual space of $W^{s,p}_0(\Omega)$) is defined as $(-\Delta_p)^s = \frac{1}{p}\Phi_{s,p}'$ (the Fr\'echet derivative).

The eigenvalue problem associated to this operator, is
\begin{equation}\label{eigen.sp}
\begin{cases}
(-\Delta_p)^s u = \lambda |u|^{p-2}u & \text{in }\Omega\\
u=0 & \text{in } \R^N\setminus\Omega.
\end{cases}
\end{equation}
This eigenvalue problem was studied by many authors and several properties were obtained that mimic the theory developed for the by now classical $p-$Laplace operator, namely
\begin{equation}\label{eigen.p}
\begin{cases}
-\Delta_p u = \lambda |u|^{p-2}u & \text{in }\Omega\\
u=0 & \text{on } \partial\Omega.
\end{cases}
\end{equation}

In particular, it is proved that the first eigenvalue for \eqref{eigen.sp} is given by the minimum of the associated Rayleigh quotient, i.e.
$$
\lambda^s_{1,p} = \inf_{v\in W^{s,p}_0(\Omega)}\frac{\Phi_{s,p}(v)}{\|v\|_p^p}.
$$
More generally, by means of the critical point theory, applying the concept of the Krasnoselskii genus (c.f. Section 2), a sequence of {\em variational eigenvalues} is constructed as
$$
\lambda^s_{k,p} = \inf_{G\in \G^k_{s,p}} \sup_{v\in G}\frac{\Phi_{s,p}(v)}{\|v\|_p^p},
$$
that is the analogous construction of the variational eigenvalues for \eqref{eigen.p}. This variational eigenvalues for \eqref{eigen.p} will be denoted by $\{\lambda_{k,p}\}_{k\in\N}$.

For this problem, in \cite{BPS}, the authors analyze the asymptotic behavior of problem \eqref{eigen.sp} as $s\to 1$. The main result of \cite{BPS} is
\begin{teo}[\cite{BPS}, Theorem 1.2]\label{teo.BPS}
Let $\Omega\subset \R^N$ be an open and bounded Lipschitz set. For any $1<p<\infty$ and $k\in \N$,
$$
\lim_{s\uparrow 1} \lambda^s_{k,p} = K \lambda_{k,p},
$$
where $K$ is the constant defined in Theorem \ref{gamma.BBM}.

Moreover, if $u_s$ is an eigenfunction of \eqref{eigen.sp} corresponding to the variational eigenvalue $\lambda^s_{k,p}$ and such that $\|u_s\|_p=1$, then there exists a subsequence $\{u_{s_n}\}_{n\in\N}\subset \{u_s\}_{s\in (0,1)}$ such that
$$
\lim_{n\to\infty}[u_{s_n} - u]_{t,q} = 0,\qquad \text{for every } p \le q < \infty \text{ and every } 0 < t < \frac{p}{q},
$$
where $u$ is an eigenfunction of \eqref{eigen.p} corresponding to the variational eigenvalue $\lambda_{k,p}$, such that $\|u\|_p=1$.
\end{teo}

On the other hand, in \cite{LL}, the authors analyze, among other things, the limit of \eqref{eigen.sp} for $p\to\infty$. To be precise, the authors fix $\alpha\in (0,1)$ and consider the problem where $s = s_p=\alpha - \frac{N}{p}$ and go on studying the limit for $p\to\infty$. 

The result in \cite{LL} for this problem reads as follows.
\begin{teo}[\cite{LL}, Proposition 20]\label{teo.LL}
Let $\Omega\subset\R^N$ be an open and bounded Lipschitz set, let $\alpha\in (0,1)$ be fixed and let $s=s_p=\alpha - \frac{N}{p}$. Define
$$
\Lambda_{1,\infty}^\alpha := \inf_{\phi\in C^\infty_c(\Omega)} \frac{\left\|\frac{\phi(x)-\phi(y)}{|x-y|^\alpha}\right\|_{L^\infty(\R^N\times \R^N)}}{\|\phi\|_\infty}.
$$
Then
$$
\lim_{p\to\infty} (\lambda^{s_p}_{1,p})^\frac{1}{p} = \Lambda_{1,\infty}^\alpha.
$$
Moreover, $\Lambda_{1,\infty}^\alpha=R^{-\alpha}$ where $R$ is the inner radius of $\Omega$.
\end{teo}

Moreover, in \cite{LL}, the authors analyze the behavior of higher eigenvalues, in the sense that they study the case of sign changing eigenfunctions for \eqref{eigen.sp}, but the authors did not get any result of the asymptotic behavior of the variational eigenvalues $\lambda^s_{k,p}$.

As a consequence of our general result, Theorem \ref{main.teo}, we can recover Theorem \ref{teo.BPS} and also Theorem \ref{teo.LL}. Moreover, in the case of Theorem \ref{teo.LL} we can also obtain the limit as $p\to\infty$ for the rest of the sequence of the variational eigenvalues.

\begin{proof}[Proof of Theorem \ref{teo.BPS} using Theorem \ref{main.teo}]

Let $\{s_n\}_{n\in\N}$ be a sequence of real numbers such that $0<s_n\uparrow 1$. Let us define $F_n\colon L^1(\Omega)\to [0,\infty]$ as
$$
F_n(v) = \begin{cases}
(1-s_n)^\frac{1}{p} [v]_{s_n, p} & \text{if } v\in W^{s_n, p}_0(\Omega)\\
\infty & \text{otherwise}.
\end{cases}
$$
We need to check that $\{F_n\}_{n\in\N}$ satisfies hypotheses (A1)--(A3). Observe that (A1) and (A2) are trivial, so we are left with (A3). To this end, we define $F_0\colon L^1(\Omega)\to [0,\infty]$ as
$$
F_0(v) = \begin{cases}
K^\frac{1}{p} \|\nabla v\|_{p} & \text{if } v\in W^{1, p}_0(\Omega)\\
\infty & \text{otherwise}.
\end{cases}
$$
Then Theorem \ref{gamma.BBM} gives us that $F_n\stackrel{\Gamma}{\to} F_0$ in $L^p(\Omega)$.

Now, just observe that Theorem \ref{main.teo} together with Theorem \ref{teo.gamma.2} directly imply that
$$
\lim_{n\to\infty} (\lambda^{s_n}_{k,p})^\frac{1}{p} = (K\lambda_{k,p})^\frac{1}{p}.
$$

The convergence of the eigenfunctions follows as in \cite{BPS}.
\end{proof}

We now turn our attention to the case $p\to\infty$. To this end, we consider a sequence $\{p_n\}_{n\in\N}$ such that $1<p_n\to\infty$, $\alpha\in (0,1)$, define $s_n = \alpha - \frac{N}{p_n}$ and define the functionals $F_n\colon L^1(\Omega)\to [0,\infty]$ as
\begin{equation}\label{Fn}
F_n(v) = \begin{cases}
(1-s_n)^\frac{1}{p_n} [v]_{s_n,p_n} & \text{if } v\in W^{s_n,p_n}_0(\Omega)\\
\infty & \text{otherwise}.
\end{cases}
\end{equation}
Next, for any $v\in C_0(\Omega)$ we denote
$$
D^\alpha v(x,y) = \frac{v(x)-v(y)}{|x-y|^\alpha},
$$
the H\"older quotient of order $\alpha$ and define $F_0\colon L^1(\Omega)\to [0,\infty]$ as
\begin{equation}\label{F0}
F_0(v) = \begin{cases}
\|D^\alpha v\|_{L^\infty(\R^N\times\R^N)} & \text{if } v\in C_0^\alpha(\Omega)\\
\infty & \text{otherwise.}
\end{cases}
\end{equation}
Let us first prove that $F_n\stackrel{\Gamma}{\to} F_0$.
\begin{lema}\label{lema.gamma}
With the above notations and hypotheses, we have that $F_n\stackrel{\Gamma}{\to} F_0$ in $C_0(\Omega)$.
\end{lema}

\begin{proof}
First, observe that since $p_n\to\infty$, then $s_n\to\alpha$ and $(1-s_n)^\frac{1}{p_n}\to 1$.

Moreover, by the definition of $s_n$, it follows that
$$
[v]_{s_n,p_n} = \|D^\alpha v\|_{L^{p_n}(\R^N\times\R^N)},
$$
for any $v\in W^{s_n,p_n}_0(\Omega)$.

Let us first check the limsup inequality. So let $v\in C_0(\Omega)$ and we need to find $v_n\in C_0(\Omega)$ such that $v_n\to v$ uniformly and $F_0(v)\ge \limsup_{n\to\infty} F_n(v_n)$.

Without loss of generality, we may assume that $F_0(v)<\infty$, therefore $v\in C_0^\alpha(\Omega)$. Let us see that if $\alpha p_n>N$, then $v\in W^{s_n, p_n}_0(\Omega)$. 

{\em Claim:} if $\alpha p>N$, then $D^\alpha v\in L^p(\R^N\times\R^N)$. 

In fact,
\begin{align*}
\iint_{\R^N\times \R^N}|D^\alpha v|^p \, dxdy & = \left ( \iint_{\Omega\times\Omega}+2\iint_{\Omega\times\Omega^c}+\iint_{\Omega^c\times\Omega^c }\right ) |D^\alpha v|^p\, dxdy\\
&=I+II+III.
\end{align*}
We trivially have that $III=0$, Next, we observe that 
$$
I\leq \|D^\alpha v\|_\infty^p |\Omega|^2.
$$
It remains to get a bound for II. But,
$$
\iint_{\Omega\times\Omega^c}|D^\alpha v|^p\, dxdy=\iint_{\Omega\times\Omega^c} \frac{|v(x)|^p}{|x-y|^{\alpha p}}\, dxdy =\int_{\Omega}w(x) |v(x)|^p\,dx,
$$
where 
$$
w(x)=\int_{\Omega^c}\frac{1}{|x-y|^{\alpha p}}\,dy\le \int_{\{|z|>d_\Omega(x)\}} \frac{dz}{|z|^{\alpha p}} = \frac{C_N}{(\alpha p-N) d_\Omega(x)^{\alpha p-N}},
$$
with $d_\Omega(x)= \text{dist}(x,\Omega^c)$. Also, since $v\in C^\alpha_0(\Omega)$, if $y\in \partial\Omega$ is such that $d_\Omega(x)=|x-y|$,
$$
|v(x)| = |v(x)-v(y)|\le \|D^\alpha v\|_\infty |x-y|^\alpha = \|D^\alpha v\|_\infty d_\Omega(x)^\alpha.
$$
Then 
$$
\int_\Omega w(x) |v(x)|^p\,dx\le C \|D^\alpha v\|_\infty^p, 
$$
where $C$ depends on $N, \alpha, p$ and $\Omega$. Hence, we arrive at
$$
\|D^\alpha v\|_p\leq C \|D^\alpha v\|_\infty < \infty,
$$
and so the claim follows.

With the help of the claim, we may take $v_n=v$ for the limsup inequality, and since $\|D^\alpha v\|_{p_n}<\infty$ if $n\ge n_0$, it follows that
$$
\lim_{n\to\infty} [v]_{s_n,p_n} = \lim_{n\to\infty} \|D^\alpha v\|_{p_n} = \| D^\alpha v\|_\infty.
$$ 

It remains to check the liminf inequality. To this end, let $v\in C_0(\Omega)$ and $v_n\in C_0(\Omega)$ be such that $v_n\to v$ uniformly. We need to see that
$$
F_0(v)\le \liminf_{n\to\infty} F_n(v_n).
$$

Without loss of generality, we may assume that $\sup_{n\in\N} F_n(v_n) <\infty$. This implies, to begin with, that $v_n\in W^{s_n,p_n}_0(\Omega)$, and hence, $D^\alpha v_n\in L^{p_n}(\R^N\times\R^N)$.

Now, let $a<\|D^\alpha v\|_\infty$ and observe that, since $v_n\to v$ uniformly, it follows that $D^\alpha v_n\to D^\alpha v$ point-wise and then
$$
\{|D^\alpha v|>a\}\subset \bigcup_{n\in\N}\bigcap_{j\ge n} \{|D^\alpha v_j|>a\} = \bigcup_{n\in\N} E_n.
$$
Observe that $E_n\subset E_{n+1}$ and that $E_n\subset \{|D^\alpha v_n|>a\}$. Therefore,
$$
|\{|D^\alpha v|>a\}|\le \lim_{n\to\infty} |E_n| \le \liminf_{n\to\infty} |\{|D^\alpha v_n|>a\}|.
$$
So we conclude that
$$
1\le \liminf_{n\to\infty} |\{|D^\alpha v_n|>a\}|^\frac{1}{p_n}.
$$
Now we use Chebyshev's inequality to obtain
$$
a^{p_n} |\{|D^\alpha v_n|>a\}|\le \|D^\alpha v_n\|_{p_n}^{p_n},
$$
From where it follows that
$$
a\le \liminf_{n\to\infty}\|D^\alpha v_n\|_{p_n} = \liminf_{n\to\infty} [v_n]_{s_n, p_n}.
$$
We now may take $a\uparrow \|D^\alpha v\|_\infty$ and the proof is complete.
\end{proof}

Now we are is position to apply Theorem \ref{main.teo} to the case $p\to\infty$.
\begin{teo}\label{p.to.infty}
Under the notation of Section 3, we define
$$
\Lambda^\alpha_{k,\infty} := \inf_{G\subset \G^k_{\alpha,\infty}}\sup_{v\in G} \frac{\|D^\alpha v\|_\infty}{\|v\|_\infty}.
$$
Then, for every $k\in\N$ we have that 
$$
\lim_{p\to\infty} (\lambda^{s_p}_{k,p})^\frac{1}{p} = \Lambda^\alpha_{k,\infty}.
$$
\end{teo}

\begin{remark}
Theorem \ref{p.to.infty} extends Theorem \ref{teo.LL} to the entire sequence of variational eigenvalues.
\end{remark}

\begin{proof}
Using Lemma \ref{lema.gamma}, it is now immediate to see that the sequence $\{F_n\}_{n\in\N}$ verifies the hypotheses (A1)--(A3) of Theorem \ref{main.teo}, and so the conclusions of Theorem \ref{main.teo} together with Theorem \ref{teo.gamma.2} imply the desired result.
\end{proof}

\subsection{Homogenization result for eigenvalues of fractional $p-$Laplace type operators}

Homogenization results for fractional operators is a subject where there are not many results in the literature. We may recall the results \cite{FB-R-S, Focardi, Focardi2,  Piatnitski-Zhizhina, Schwab, Schwab2}.

So here, given $0<\alpha\le \beta <\infty$, we consider the class of kernels 
$$
\A_{\alpha,\beta} := \{a\in L^\infty(\R^N\times\R^N)\colon a(x,y)=a(y,x) \text{ and } \alpha\le a(x,y)\le \beta \text{ a.e.}\}.
$$
Associated to each $a\in \A_{\alpha,\beta}$ we consider the operator
$$
\L_a u(x) = \L^{s,p}_a u(x) = \text{p.v.} \int_{\R^N} a(x,y)\frac{|u(x)-u(y)|^{p-2}(u(x)-u(y))}{|x-y|^{N+sp}}\, dy.
$$
This operator turns out to be the Fr\'echet derivative of the functional
$$
\Phi_a\colon W^{s,p}_0(\Omega)\to\R,\quad \Phi_a(u) := \iint_{\R^N\times\R^N} a(x,y)\frac{|u(x)-u(y)|^p}{|x-y|^{N+sp}}\, dxdy.
$$

Consider now a sequence of kernels $\{a_n\}_{n\in\N}\subset \A_{\alpha,\beta}$. Without loss of generality (passing to a subsequence), we may assume the existence of $a_0\in \A_{\alpha,\beta}$ such that $a_n\stackrel{*}{\rightharpoonup} a_0$ weakly* in $L^\infty$.

For $n\in \N_0$, we define $F_n\colon L^1(\Omega)\to [0,\infty]$ as
$$
F_n(v) := \begin{cases}
\left(\Phi_{a_n}(v)\right)^\frac{1}{p} & \text{if } v\in W^{s,p}_0(\Omega)\\
\infty & \text{otherwise}.
\end{cases}
$$

It is proved in \cite{Focardi} that $F_n\stackrel{\Gamma}{\to} F_0$ in $L^p(\Omega)$. However, we include a proof of this fact in order to keep the paper self contained. 

\begin{prop}
Let $F_n$ be the functionals defined above, then $F_n\stackrel{\Gamma}{\to} F_0$ in $L^p(\Omega)$.
\end{prop}

\begin{proof}
Let us begin with the liminf inequality. Take $v\in L^p(\Omega)$ and $\{v_n\}_{n\in\N}\subset L^p(\Omega)$ be such that $v_n\to v$ in $L^p(\Omega)$ and 
$$
F_0(v)\le \liminf_{n\to\infty} F_n(v_n).
$$
We may assume that $\sup_{n\in\N} F_n(v_n)<\infty$ and so $v_n\in W^{s,p}_0(\Omega)$ for every $n\in\N$. Moreover, this assumption also implies that $\{v_n\}_{n\in\N}$ is bounded in $W^{s,p}_0(\Omega)$.

Passing to a subsequence, if necessary, we can assume that $v_n\rightharpoonup v$ weakly in $W^{s,p}_0(\Omega)$.

Consider now $0<\delta<R<\infty$ and define the sets
$$
Q_{R,\delta} = (B_R(0)\times B_R(0))\setminus \{(x,y)\in \R^N\times\R^N\colon |x-y|\le \delta\}.
$$
Since $|x-y|^{-(N+sp)}$ is bounded in $Q_{R,\delta}$ and $|v_n|^p\to |v|^p$ strongly in $L^1(\R^N)$, it follows that
$$
\frac{|v_n(x)-v_n(y)|^p}{|x-y|^{N+sp}}\to \frac{|v(x)-v(y)|^p}{|x-y|^{N+sp}} \quad \text{strongly in } L^1(Q_{R,\delta}).
$$
Therefore,
\begin{align*}
\liminf_{n\to\infty} F_n(v_n)^p &= \liminf_{n\to\infty}\iint_{\R^N\times\R^N} a_n(x,y) \frac{|v_n(x)-v_n(y)|^p}{|x-y|^{N+sp}}\, dxdy\\
&\ge \lim_{n\to\infty}\iint_{Q_{R,\delta}} a_n(x,y) \frac{|v_n(x)-v_n(y)|^p}{|x-y|^{N+sp}}\, dxdy\\
&=\iint_{Q_{R,\delta}} a_0(x,y) \frac{|v(x)-v(y)|^p}{|x-y|^{N+sp}}\, dxdy.
\end{align*}
Now, taking the limit as $R\to\infty$ and $\delta\to 0$ we obtain the liminf inequality.

Next we turn our attention to the limsup inequality. So let $v\in L^p(\Omega)$ and we must find a recovery sequence $\{v_n\}_{n\in\N}\subset L^p(\Omega)$ such that $v_n\to v$ in $L^p(\Omega)$ and $F_0(v)\ge \limsup_{n\to\infty} F_n(v_n)$. Since we can assume that $F_0(v)<\infty$, it follows that $v\in W^{s,p}_0(\Omega)$ and this implies that 
$$
\frac{|v(x)-v(y)|^p}{|x-y|^{N+sp}}\in L^1(\R^N\times\R^N).
$$

Hence, we may take $v_n=v$ and, since $a_n\stackrel{*}{\rightharpoonup} a_0$ weakly* in $L^\infty(\R^N\times\R^N)$ we get
\begin{align*}
\lim_{n\to\infty}F_n(v)^p &= \lim_{n\to\infty}\iint_{\R^N\times\R^N} a_n(x,y) \frac{|v(x)-v(y)|^p}{|x-y|^{N+sp}}\, dxdy\\
&=\iint_{\R^N\times\R^N} a_0(x,y) \frac{|v(x)-v(y)|^p}{|x-y|^{N+sp}}\, dxdy\\
&= F_0(v)^p.
\end{align*}
This completes the proof.
\end{proof}

Therefore, as an immediate consequence of Theorem \ref{main.teo}, we get
\begin{teo}
Under the notations of this subsection, if we denote
$$
\lambda_k^n = \sup_{G\in \G^k_{s,p}} \inf_{v\in G} \frac{F_n(v)}{\|v\|_p},
$$
for $n\ge 0$, then
$$
\lim_{n\to\infty}\lambda_k^n = \lambda_k^0,\quad \text{for all }k\in\N.
$$
\end{teo}

\begin{remark}
Observe that each $\lambda_k^n$ is an eigenvalue for the problem
$$
\begin{cases}
\L_{a_n} u = \lambda_k^n |u|^{p-2}u & \text{in }\Omega\\
u=0 & \text{in }\Omega^c
\end{cases}
$$
for $n\ge 0$.
\end{remark}

\begin{proof}
The proof is immediate from Theorems \ref{main.teo} and \ref{teo.gamma.2}, since hypotheses (A1)--(A3) are trivially verified.
\end{proof}

\begin{remark}
Up to our knowledge, this eigenvalue problem has not been considered previously in the literature.
\end{remark}

\section*{Acknowledgements}

Supported by Universidad de Buenos Aires under grant UBACYT Prog. 2018 20020170100445BA,
by ANPCyT under grants PICT 2016-1022 and PICT 2017-0704, by Universidad Nacional de San Luis under grants PROIPRO 03-2418 and PROICO 03-1916 and by CONICET under grant PIP 11220150100032CO. 

J. Fern\'andez Bonder and Anal\'ia Silva are members of CONICET.

\bibliographystyle{amsplain}
\bibliography{biblio}

\end{document}